\documentclass[12pt,oneside,reqno]{amsart}
\usepackage{}
\usepackage{txfonts}
\usepackage{bbm}
\usepackage{cases}
\usepackage{amsmath}
\usepackage{graphicx}
\usepackage{mathrsfs}
\usepackage{stmaryrd}
\usepackage{amsfonts}
\usepackage{xcolor}
\usepackage{enumerate,amsmath,amssymb,amsthm}

\oddsidemargin=0mm \topmargin=-12mm \numberwithin{equation}{section}

\makeatletter

\newcommand{\Rmnum}[1]{\expandafter\@slowromancap\romannumeral #1@}
\makeatother

\numberwithin{equation}{section}
\newtheorem{thm}{Theorem}[section]

\newtheorem{corollary}{Corollary}[section]
\newtheorem{lem}{Lemma}[section]
\newtheorem{definition}{Definition}[section]
\newtheorem{Rem}{Remark}[section]
\newtheorem{example}{Example}[section]

\def\eps{\varepsilon}

\def\ti{\tilde}

\def\<{{\langle}}
\def\>{{\rangle}}
\def\({{\Big(}}
\def\){{\Big)}}
\def\]{{\Big]}}
\def\[{{\Big[}}

\def\bx{{\mathbf{x}}}

\def\dif{{\mathord{{\rm d}}}}

\def\min{{\mathord{{\rm min}}}}

\def\={&\!\!=\!\!&}
\def\bt{\begin{theorem}}
\def\et{\end{theorem}}
\def\bl{\begin{lemma}}
\def\el{\end{lemma}}
\def\br{\begin{remark}}
\def\er{\end{remark}}

\def\bd{\begin{definition}}
\def\ed{\end{definition}}
\def\bp{\begin{proposition}}
\def\ep{\end{proposition}}
\def\bc{\begin{corollary}}
\def\ec{\end{corollary}}
\def\bx{\begin{Examples}}
\def\ex{\end{Examples}}

\def\cC{{\mathcal C}}
\def\cD{{\mathcal D}}

\def\cF{{\mathcal F}}

\def\cP{{\mathcal P}}

\def\mE{{\mathbb E}}

\def\mI{{\mathbb I}}

\def\mN{{\mathbb N}}

\def\mP{{\mathbb P}}

\def\mR{{\mathbb R}}

\def\mZ{{\mathbb Z}}

\def\sB{{\mathscr B}}

\def\geq{\geqslant}
\def\leq{\leqslant}

\numberwithin{equation}{section}

\title{\bf{Exponential Mixing for SDEs under the total variation$^*$}\thanks{The authors
were Supported by 973 Program, No. 2011CB808000 and
Key Laboratory of Random Complex Structures and Data Science, No.2008DP173182, NSFC, No.:10721101, 11271356, 11371041.}}

{\author{   Xuhui Peng  \\
{\em\tiny  College of Mathematics and Computer Science,Hunan Normal University, 
\\
Changsha, Hunan 410081, P. R. China
 }\\
}
\author{Rangrang  Zhang$^\dag$  \\
{\em\tiny  Department of  Mathematics,
Beijing Institute of Technology, Beijing, 100081, China.
 }\\
}
}

\date{}
\begin{document}
\footnote{ Email addresses:  xhpeng@hunnu.edu.cn,  rrzhang@amss.ac.cn}
\footnote{$^*$ Xuhui Peng
was  supported by NSFC (No.11501195), a Scientific Research Fund of Hunan Provincial Education Department (No.17C0953),
  the Youth Scientific Research Fund of Hunan Normal University (No.Math140650) and  the Construct Program of the Key Discipline in Hunan Province. }

\maketitle
\begin{abstract}
We establish a  general criterion which ensures  exponential  ergodicity   of
Markov process on $\mathbb R^d$. %under the total variation.
Compared  with   the  classical irreducible condition,
we  only require
 a weak   form of  irreducibility   given by Hairer and Mattingly
    [\emph{Annals of Probability} \textbf{36(6)} (2008) 2050-2091].
     Applying  our  criterion  to
 stochastic differential equations driven by  L\'evy noise, we obtain the exponential ergodicity. Our noise can be more degenerate  than the existing   results.

  \vspace{1mm}
\end{abstract}

\section{Introduction}

%\subsection{Motivations}
The objective of  this  paper  is to  study  exponential  ergodicity of  Markov process on $\mathbb R^d$ under the total variation. There are many works  on this topic, we only mention some of them which are related to our work.

 In 1990s, Meyn and Tweedie  established  a framework for  ergodicity in \cite{MST3}.
 %MST1,MST2,MST4,
%Down et al \cite{DM} developed a similar  theory  for  $\psi$-irreducible  continuous process.
% In their  framework,
For $\psi$-irreducible and aperiodic Markov chain $\{X_t, t\in \mZ^{+}\}$   on a general state space $E$, they obtained the  exponentially  ergodic property  %under the conditions:
 if there exist  a petite set $\cC$, constants $\alpha<1,\beta<\infty$,  a  time $t=1$ and a function  $V\geq 1$ finite  at some  one
$x_0\in E$ satisfying
\begin{eqnarray}\label{z-3}
  P_tV(x) \leq  \alpha   V(x)+\beta    I_{\cC}(x),\quad x\in E.
\end{eqnarray}
% \begin{eqnarray*}
%   \int_{E}P(x,\dif y)V(y)-V(x)\leq  -\beta V(x)+b \mathbbold{1}_{\cC}(x),\quad x\in E.
% \end{eqnarray*}
More details are referred to   \cite[Theorem 15.0.1]{MST3}.
% including definitions of $\psi$-irreducible,  aperiodic, petite set etc.
For  general     Markov process $\{X_t,t\in \mR^+\}$, under a similar condition,
Down, Meyn and Tweedie \cite[Thoerem 5.2]{DM}  obtained   exponentially  ergodic property.

% There  are   two fundamental assumptions in the Meyn-Tweedie \cite{MST3} framework for ergodicity.  The first is the existence of  a  Lyapunov function,  which  ensures  that dynamics will  move inward average outside some compact region $\cC$. The second  is the existence of  a neighborhood $\cN$ of some distinguished points in $\cC$,  which is uniformly reachable from inside $\cC$ and the probability densities are smooth  in $\cC$.
% For   accurate meaning of the second  fundamental assumption, we refer the readers to \cite[Assumption 2.1]{MSH}.
 %  Employing   these  two fundamental assumptions,

Following the ideas  in \cite{MST3},
 Mattingly  et al.  \cite[Theorem 2.5]{MSH}  obtained   exponentially  ergodic property  for Markov chain or process  $X_t$  on  $(\mR^d,\sB(\mR^d))$ under  the following three  hypotheses:
\begin{itemize}
\item[$(\textbf{M}_1)$]  For some  fixed  compact set $ \cC \in\sB(\mR^d)$ and some  $y^*\in int(\cC)$, there is, for any $\delta>0$, a $t_0=t_0(\delta)>0$ such that
 \begin{eqnarray*}
   P_{t_0}(x,B_\delta(y^*))>0,   \quad x\in \cC,
 \end{eqnarray*}
 where   $B_\delta(y^*):=\{u\in \mR^d, |u-y^*|<\delta\}$;
 \item[$(\textbf{M}_2)$]
 For $t>0$, the transition kernel possesses a  density $p_t(x,y)$, precisely
     \begin{eqnarray*}
       P_t(x,A)=\int_A p_t(x,y)\dif y, \forall x\in \cC, A\in \sB(\mR^d)\cap \sB(\cC)
     \end{eqnarray*}
 and $p_t(x,y)$ is  jointly continuous in $(x,y)\in \cC\times \cC,$ where
 $\cC$ is the same as that  in Hypothesis $\textbf{M}_1;$
\item[$(\textbf{LC})$]
%Let $P_t$ be a Markov semigroup. One says that
Lyapunov condition holds for $P_t$, that is,  there is a function $V:\mR^d\rightarrow [1,\infty)$  with $\lim_{x\rightarrow \infty}V(x) = +\infty$ such that for some $t_*>0,$  and some real numbers $\alpha \in (0, 1)$, $\beta \in [0,\infty)$,
    \begin{eqnarray*}
    && P_{t_*}V(x)\leq \alpha  V (x) + \beta.
 \end{eqnarray*}
\end{itemize}
Their   result can be applied to a number of   SDEs. We refer  % the  readers
 to \cite{MSH,MS,Za} etc.
%\textcolor[rgb]{1.00,0.00,0.00}{Their result is easily applied to SDEs.}
%Their result  can be   applied  to a variety of SDEs.

% proof is essentially  that of  Meyn and Tweedie \cite{MST3}. However, they employ a particular but useful, reachable  structure  which can

 % proved a ergodic  theorem
% which is general enough to apply to a variety of SDEs.

%[Theorem 2.5]
We also  mention  the   ergodic theorem given  by  Rey-Bellet \cite{LR}.
%In \cite{LR}, Rey-Bellet
%We also  mention  the   ergodic theorem   in the very readable  lecture notes  given  by Rey-Bellet \cite{LR}. In his   lecture notes,  he
He obtained an  exponential ergodicity
if the  Hypotheses $\textbf{LC}$ and
\begin{itemize}
\item[$(\textbf{H}_a)$] The Markov process $X_t$ is irreducible aperiodic, i.e there exists a $t_0>0$ such that
$$
P_{t_0}(x,A)>0, \   \forall x \in \mathbb{R}^d, \forall \text{ open  set } A,
$$
\item[$(\textbf{H}_b)$]   For $t>0$, the transition kernel possesses a  density $p_t(x,y)$
 which is a  smooth  function of $(x,y),$ in particular, $P_t$ is strong Feller,
\end{itemize}
hold.
% Obviously, Hypotheses $\textbf{H}_a$ and   $\textbf{H}_b$ are stronger than
%$\textbf{MSH}_1$ and $\textbf{MSH}_2.$

There are other ways to obtain  exponential ergodicity. % using other methods,
We refer   to   \cite{BGL,B,L} etc.  for  log-Sobolev  or hypercontractivity estimates and \cite{WJ,Zhang} etc.   for coupling method.

 %such as the coupling method, where  log-Sobolev  or hypercontractivity estimates are used (cf. \cite{BGL,B,L},(cf. \cite{WJ,Zhang} etc.).

%frameworks to work within. There are assumptions which lead to log-Sobolev  or hypercontractivity estimates  (cf. \cite{BGL,B,L}, etc.). Using  the coupling method is another   way  to  obtain the exponential ergodicity  (cf. \cite{WJ,Zhang}, etc.).
%

% We consider a time homogeneous Markov
% process $X$ living  in  $\mR^d$,  and  write
% \begin{align*}
%  X_t=X_t(x),
% \end{align*}
% where $x\in \mR^d$ is the initial value.  We denote by $\{P_t\}_{t\in \mR^{+}}$ the Markov semi-group associated Markov famly $(X_{\cdot}(x))_{x\in \mR^d}$ and use $P_t(x, \cdot), t \in \mR^{+}
% , x \in \mR^d$ to denote   the transition function for the process $X$.
% In this article,    the process $X$
% is supposed to be strong Markov, be  adapted to a filtered  probability space $(\Omega,\cF,(\cF_t)_{t\geq 0}$ $,\mP)$ and to have c\'adl\'ag trajectories.

In this paper, we present   a new criterion to  ensure  exponential  ergodicity   of
Markov process on $\mathbb R^d$,
in which  our   hypotheses  are weaker and  easier to  verify  than Hypotheses $\textbf{H}_a,$ $\textbf{M}_1$
in some situations.
To show this,  an example (Example \ref{E-1} below) is given. In our criterion,
we  only require
 a weak   form of  irreducibility  given by Hairer and Mattingly \cite{martin,HM}
 instead of   the classical  irreducible condition (Hypothesis $\textbf{H}_a$).
As an application, we apply our criterion to stochastic differential equations driven by  L\'evy noise and  obtain its exponential mixing property under the total variation. However, it is not easy  to verify Hypothesis  $\textbf{M}_1$ or
apply the Meyn-Tweedie framework \cite{MST3} in our cases.  Moreover, compared with the results in \cite{Kulik,Q,WJ,WJ2013}, our noises can be  more  degenerate.

%\subsection{Layout and notations}
This paper is organized as follows: In Section 2,   we present   a new criterion.
In Section 3, we give a proof of the   criterion.
In Section 4, we  apply our criterion  to stochastic differential equations driven by  L\'evy noise.
%and    obtain an  exponential mixing under  the total variation
 % Compared to the results in \cite{Kulik,Q,WJ} etc., our noises  can be  more degenerate.
   % Compared  to  the results in \cite{Xu}, we obtain the exponential mixing    under a  stronger norm.
% \end{Rem}
% Our noise can be more degenerate than   existed  results.

Before concluding this introduction, we collect some notations  and make some conventions for later use.
\begin{itemize}
% \item $\bullet$
\item $\nabla=(\partial_1,\cdots,\partial_d)$ denotes the gradient operator.
\item $\sB(\mR^d) $  denotes the collection of all Borel measurable sets on $\mR^d$  and   $\sB_b(\mR^d)$  denotes    the space of bounded and Borel measurable functions.
\item  For any $f\in \sB_b(\mR^d)$,
  $\|f\|_\infty:=ess.\sup_{x\in \mR^d}|f(x)|.$
  \item For any $x=(x_1,\cdots,x_d),y=(y_1,\cdots,y_d)\in\mR^d,$ $|x|:=\sqrt{\sum_{i=1}^dx_i^2}, ~~|(x,y)|:=\sqrt{|x|^2+|y|^2}.$
  \item   $\cD(X)$ denotes  the law of  random variable  $X$.
  \item  For  a  sign measure $\mu$ on $(\mR^d,\sB(\mR^d)$, $\mu(f):=\int_{\mR^d}f(y)\mu(\dif y)$ and
       $$ \|\mu \|_{var}:=\sup\{|\mu(\Gamma)| ~~~~:~~\Gamma\in\sB(\mR^d)\}.$$
  \item For any  $R>0,$  $B_{R}:=\{u\in \mR^d, |u|< R\},$ $\overline{B}_{R}:=\{u\in \mR^d, |u|\leq R\}.$
  \item The letter $C$ denotes an unimportant constant, whose value  may change  in different places.
    \item   When the initial value of  Markov  process $X_t$ is $x\in \mR^d$, we also denote this process  by $X_t(x).$
  \item For any $x\in \mR^d,$ $\delta_x$  denotes  the Dirac measure
concentrated  on  $x$.
 \end{itemize}

%\subsection{Main Results}
\section{ A general criterion}
This section is devoted to the statement  of a general criterion.
Let $(\Omega,\cF,(\cF_t)_{t\geq 0}$ $,\mP)$ be a filtered  probability space,  and  $\{X_t, t\in \mR^+\}$ be a strong Markov process on $\mathbb R^d$ that  is  adapted to this  filtered  probability space and  supposed to have c\'adl\'ag trajectories.
Let  $\{P_t\}_{t\in \mR^{+}}$ be  the Markov semi-group associated with the  process $X$, $P_t(x, \cdot)$  be      the transition kernel for  $X$.

% We say $\mu$ is a invariant measure for $P_t$, if,
% \begin{eqnarray*}
% \mu(A)=P_t^*\mu(A):=\int_{\mR^d}P_t(x,A)\mu(dx), ~~~~\forall t>0, ~\forall A \in \sB(\mR^d).
% \end{eqnarray*}

%Assume  $F(u)\geq 1$ is   a continuous  function  on $\mR^d$ tending to $+\infty$ as $|u|\rightarrow \infty$.
Our Hypotheses in this article    are
\begin{itemize}
  % \item $\textbf{H}_1$(Lyapunov function condition). There are
  %  positive constants   $t_*,R_*,C_*$ and $a<1$ such that
%  \begin{eqnarray*}
%   P_{t_*}F(x)  &\leq &aF(x), \ \ \ \  for~ |x|\geq R_*,
%  \\ P_tF(x)  &\leq &C_*, \ \ \ \  for~ |x|\leq R_*, \forall t\geq 0.
%  \end{eqnarray*}
%  \end{itemize}
 %  \begin{itemize}
  \item[$(\textbf{H}_1)$]  (Weak form of irreducibility).
  For any $R,\delta>0,$
  there exist positive constants $R_0:=R_0(R)>0$
 and  $T_0:=T_0(R,\delta)$ such that for any $t \geq T_0$ and any  $x,y\in \overline{B}_{R}$,
\begin{eqnarray*}
  \sup_{\pi\in \Gamma(P_t^*\delta_{x}, P_t^*\delta_{y})}\pi\left\{(x',y')\in \mR^d\times \mR^d, |x'-y'| \leq \delta, ~~~~x',y'\in \overline{B}_{R_0}\right\}>0,
\end{eqnarray*}
where
$P_t^*$ is the semigroup acting  on probability measures which is dual to $P_t$  and  $\Gamma(\mu,\nu)$ denotes the set of  probability measures $\pi$ on $\mR^d\times \mR^d$ such that
$$\pi(A\times \mR^d)=\mu(A),\pi(\mR^d\times A)=\nu(A), \forall A\in \sB(\mR^d).$$
  \item[$(\textbf{H}_2)$]
For any  $t>0,x\in \mR^d,$ we have
  \begin{eqnarray}\label{x-6}
    \lim_{y\rightarrow x} \sup_{\|f\|_{\infty}\leq 1}\[P_tf(x)-P_tf(y)\]=0.
  \end{eqnarray}
  %or
%
\end{itemize}

\begin{Rem}
Hypothesis $\textbf{H}_2$ holds if we have the following gradient estimate
\begin{eqnarray*}
    |\nabla P_tf(x)|\leq C_t(|x|)\|f\|_{\infty},~~~~~~~ \forall f\in \mathscr{B}_b(\mR^d),
 \end{eqnarray*}
  where  $C_t(\cdot)$  is a  locally bounded function from  $\mR^{+}$ to $\mR^{+}$ for fixed $t.$
\end{Rem}
 We now state our criterion,  its proof will be given   in  Section 3.
\begin{thm}\label{9}
Assume Hypotheses  $\textbf{LC}, \textbf{H}_1,\textbf{H}_2$  hold,   then the  process $X_t$ is  exponentially   ergodic under the total variation, i.e, there exist a unique invariant probability measure $\mu$ for $P_t$ and some positive   constants  $\theta,C$ such that for any $x\in \mR^d$,
\begin{eqnarray}\label{16}
  \|P_t(x,\cdot)-\mu\|_{var}\leq Ce^{-\theta t}(1+V(x)),\ \ \ \forall t\geq 0.
\end{eqnarray}
\end{thm}

Obviously, our Hypothesis  $\textbf{H}_1$ are weaker than  $\textbf{H}_a$ and  $\textbf{M}_1$.
Although   Hypothesis $\textbf{H}_2$  is stronger than   $\textbf{H}_b$,  there is no
essential  differences in   verification   for stochastic differential equations. Specifically,
  the usual sufficient  condition for   Hypothesis  $\textbf{H}_b$ is  H\"ormander condition in Brownian case  or some  similar condition in L\'evy case, for example, \cite{DPSZ,Ichihara,SZ} etc. Under these conditions, $\textbf{H}_2$ also holds. For more details, one can see \cite[Theorem 4.2]{DP1} for the  Brownian case, \cite[Remark 1.2 and  Theorem 1.3]{DPSZ} and  \cite[Theorem 1.1]{SZ}  for the  L\'evy case.

% Under  H\"ormander condition or some  similar condition, \textcolor[rgb]{1.00,0.00,0.00}{the Malliavin matrix of the solution to SDEs is almost invertible (cf. \cite{Nualart}). Once the Malliavin matrix is almost invertible, Hypothesis  $\textbf{H}_2$ holds.  } For this, one can see \cite[Corollary 9.6.12]{Bogachev}\cite[Theorem 1.1]{DPSZ}  for more  details.

The advantage of  Hypothesis  $\textbf{H}_1$ is that, we only need to give some moment estimates such as
\begin{eqnarray}\label{x-7}
  \mE \[|X_t(x)-X_t(y)|^{p}\], \quad  \mE\[|X_t(x)|^p\], \quad  \text{here }p>0
\end{eqnarray}
in some situations,
and don't need to verify  Hypothesis  $\textbf{H}_a$, which  is always  complicated and needs control theory. One  can  see Section 3 and  \cite[Theorem 6.1]{LR} for more  details.

Our Hypotheses in  Theorem \ref{9} are  weaker  than Hypotheses  $\textbf{H}_a, \textbf{H}_b$ in some situations. To show this, we  consider the following example.
\begin{example}\label{E-1}
Fix two  constants  $k>0, \sigma \neq 0.$ Consider the following SDEs,
\begin{eqnarray}
\label{40}
\left\{
  \begin{split}
   & \dif X_t=Y_t^2\dif t-kX_t\dif t
   \\ &\dif Y_t=-kY_t\dif t+\sigma \dif W_t
   \\& (X_t,Y_t)|_{t=0}=(x,y)\in \mR\times \mR,
  \end{split}
  \right.
\end{eqnarray}
where $(W_t)_{t\geq 0}$ is a one  dimensional  Brownian motion,
 then the hypotheses in Theorem \ref{9} hold, but the  Hypothesis  $\textbf{H}_a$ doesn't hold.
\end{example}
\begin{proof}
Clearly, we have
  \begin{eqnarray}
  \label{41}
  \left\{
  \begin{split}
    Y_t&=Y_t(x,y)=e^{-kt}y+\sigma \int_0^te^{-k(t-s)}\dif W_s,
    \\ X_t&= X_t(x,y)=e^{-kt}x+\int_0^te^{-k(t-s)} Y_s^2\dif s.
    \end{split}
    \right.
  \end{eqnarray}

(\Rmnum{1})   By (\ref{41}), we obtain
$$
\mP(X_t\geq e^{-kt}x)=1.
$$
Therefore, Hypothesis  $\textbf{H}_a$ doesn't hold.

Now, we give a verification of   the  hypotheses in Theorem \ref{9}.

(\Rmnum{2})
 The verification of Hypothesis  $\textbf{LC}$. Let $V(x,y)=|y|^2+|x|$.  By (\ref{41}), we obtain
 \begin{eqnarray*}
   \mE\[V(X_t,Y_t)\]&=&\mE|Y_t|^2+\mE|X_t|
   \\ &\leq & \mE|Y_t|^2+e^{-kt}|x|+\int_0^te^{-k(t-s)} \mE Y_s^2\dif s
   \\ &\leq & \big(2e^{-2kt}|y|^2+C\big)+e^{-kt}|x|+\int_0^te^{-k(t-s)}\big(2  e^{-2ks}|y|^2+C\big) \dif s
   \\ &\leq & 2 e^{-2kt}|y|^2+C+e^{-kt}|x|+\frac{2}{k}e^{-kt}(1-e^{-kt})|y|^2+C,
\\ &\leq &e^{-kt}|x|+\big(2e^{-2kt}+\frac{2}{k}e^{-2kt}\big)|y|^2+C,
 \end{eqnarray*}
which gives the desired result.

(\Rmnum{3}) The verification  of Hypothesis  $\textbf{H}_1$.
  Let $(B_t)_{t\geq 0}$ be a Brownian motion which is independent of $(W_t)_{t\geq 0}$.   For any   $(x',y')\in \mR^2$.
  Let $(\tilde{X}_t,\tilde{Y}_t)$ be the solution to Eq.(\ref{40}) with initial value $(x',y')$ and noise $(B_t)_{t\geq 0}$,
  that is
  \begin{eqnarray}
  \label{42}
  \left\{
  \begin{split}
    \tilde{Y}_t&=\tilde{Y}_t(x',y')=e^{-kt}y'+\sigma e^{-kt}\int_0^te^{ks}\dif B_s,
    \\ \tilde{X}_t&= \tilde{X}_t(x',y')=e^{-kt}x'+\int_0^te^{-ks} \tilde{Y}_s^2\dif s.
    \end{split}
    \right.
  \end{eqnarray}
  For any  $R,\delta>0$  and   $(x,y),(x',y')\in B_{R}$, it is easy to see that   there exists  a constant  $T_0:=T_0(R,\delta)$ such that for any $t \geq T_0$,
\begin{eqnarray}
\label{44}
&& e^{-kt}|x-x'|<\delta.
\end{eqnarray}
For any $ R_0,\delta,\eps>0$ and  $ t>T_0$,
  let $h,\tilde{h}\in C([0,t],\mR)$ satisfying
\begin{eqnarray*}
&& h_0=y, \ \ \tilde{h}_0=y',
\\ &&   |h_t-\tilde{h}_t|<\frac{\delta}{2},\   |h_t|\leq \frac{R_0}{2}, \   |\ti{h}_t|\leq \frac{R_0}{2},
\\ && \int_0^te^{-ks}|h_s^2-\tilde{h}_s^2|\dif s<\frac{\delta}{8},
\end{eqnarray*}
and  denote
\begin{eqnarray*}
  B^h_\eps=\{u(\cdot)\in C([0,t],\mR):u(0)=y, \sup_{s\in [0,t]}|u(s)-h_s|< \eps \wedge 1 \wedge \frac{R_0}{2}\},
  \\  B^{\tilde{h}}_\eps=\{u(\cdot)\in C([0,t],\mR):u(0)=y', \sup_{s\in [0,t]}|u(s)-\tilde{h}_s|< \eps \wedge 1 \wedge \frac{R_0}{2}\}.
\end{eqnarray*}
Then there exists a  positive constant $\eps=\eps(t,\delta,k)$,   such that for any  $u_\cdot\in B^h_\eps,\tilde{u}_\cdot \in B^{\tilde{h}}_\eps $
\begin{eqnarray}
\label{43}
\begin{split}
& \int_0^te^{-ks}|u_s^2-\tilde{u}_s^2| \dif s<\delta,
\\ &|u_t-\tilde{u}_t|<\delta,
\\ & |u_t|<R_0,\ |\tilde{u}_t|<R_0.
\end{split}
\end{eqnarray}
 By   the definitions  of $B^h_\eps, B^{\tilde{h}}_\eps$ and (\ref{41})-(\ref{43}), we derive  that
\begin{eqnarray*}
  \nonumber && \mP(|X_t-\tilde{X}_t|<2\delta, |Y_t-\tilde{Y}_t|<2\delta, |Y_t|\leq R_0, |\tilde{Y}_t|\leq R_0 )
  \\&&\geq \mP((Y_s)_{s\in [0,t]}\in B^h_\eps, (\tilde{Y}_s)_{s\in [0,t]}\in  B^{\tilde{h}}_\eps)
  \\ &&\geq   \mP((Y_s)_{s\in [0,t]}\in B^h_\eps)\cdot  \mP( (\tilde{Y}_s)_{s\in [0,t]}\in  B^{\tilde{h}}_\eps)
  \\ &&>0,
  \end{eqnarray*}
  where in the last inequality  we have used  \cite[Theorem 3.2]{P}.
This completes the verification  of Hypothesis $\textbf{H}_1$.

(\Rmnum{4}) The verification  of Hypothesis $\textbf{H}_2$. By \cite[Theorem 4.2]{DP1}, Hypothesis  $\textbf{H}_2$ holds.

\end{proof}

In the end of this part, we  state   the main  ideas used in the proof of  Theorem  \ref{9}. A  concrete proof  is presented in Section 3.

The main tool to prove Theorem \ref{9} is the coupling method.
%  introduced by Odasso \cite{Wang}.
 In the proof of Theorem  \ref{9}, we also adopt  some ideas from \cite{Xu}.
For any $x,y\in \mR^d$, Hypothesis $\textbf{LC}$  is used to  ensure that the processes $(X_t(x), X_t(y))$ enter a ball $\overline{B}_{R_*}=\{u\in \mR^d\times \mR^d, |u|\leq R_*\}$  very quickly, see Lemmas \ref{23}, \ref{36} below for more   details. Denote $\tau$ the time of the two    processes $(X_t(x), X_t(y))$ enter this   ball $\overline{B}_{R_*}$.

Hypothesis $\textbf{H}_1$ is used to   ensure  that
  \begin{eqnarray*}
  \mP\left(
  \begin{split}& |X_{\tau+T}(x)-X_{\tau+T}(y)|~ \text{is very small}
   \\ &\text{and}  ~X_{\tau+T}(x), X_{\tau+T}(y) \text{ stay  at some ball}
  \end{split}
    \right)>0
\end{eqnarray*}
 holds  for some big but finite $T$.
 Denote $u_x=X_{\tau+T}(x), u_y=X_{\tau+T}(y).$

  By  Hypothesis $\textbf{H}_2$, one  finds a  $t_0>0$ such that (see Lemma \ref{2} below for more details),
  \begin{eqnarray*}
    \|P_{t_0}(u_x,\cdot)-P_{t_0}(u_y,\cdot)\|_{var}=\min \mP(Z_1\ne Z_2)<1,
  \end{eqnarray*}
  the minimum is taken over all couplings $(Z_1,Z_2)$ of $(P_{t_0}(u_x,\cdot),P_{t_0}(u_y,\cdot)).$ Then,  there is a coupling $(Z_1,Z_2)$ of $ (X_{t_0}(u_x),(X_{t_0}(u_y))$ such that
   \begin{eqnarray*}
 \mP(Z_1= Z_2)>0,
  \end{eqnarray*}
  which implies  with a positive probability, the coupling time is $\tau+T+t_0.$

Based on the above arguments   and by the strong Markov property,
 we obtain  a estimate  of  the coupling time   of $X_t(x)$ and $X_t(y)$ (Lemma \ref{1} below)  and   complete the proof of Theorem \ref{9}.

\section{Proof of Theorem \ref{9}}
This section is devoted to prove  the  general criterion.
Throughout this section, we assume   that
Hypotheses  $\textbf{LC},\textbf{H}_1$  and  $\textbf{H}_2$ hold.
\subsection{ Construction of the coupling Markov chain and some lemmas}

Let   $(\Lambda_1, \Lambda_2)$ be two probability measures on a metric space $(\mR^d,\sB(\mR^d))$.
 We say $(Z_1,Z_2)$
is a coupling of $(\Lambda_1, \Lambda_2)$ if $\Lambda_1=\cD(Z_1), \Lambda_2=\cD(Z_2).$

We first  recall a fundamental result in the  coupling methods.
\begin{lem}(\cite{L1} etc.)
\label{3}
  Let $(\Lambda_1, \Lambda_2)$ be two probability measures on a metric space $(\mR^d,\sB(\mR^d))$.
  Then
  \begin{eqnarray*}
    \|\Lambda_1-\Lambda_2\|_{var}=\min \mP(Z_1\ne Z_2).
  \end{eqnarray*}
  The minimum is taken over all couplings $(Z_1,Z_2)$ of $(\Lambda_1,\Lambda_2).$ There exists a coupling
  which reaches the minimum value. It is called a maximal coupling.
\end{lem}

Now,  we list some lemmas which will be used in the proof of Theorem \ref{9}.
\begin{lem}
\label{23}
(\romannumeral1)  For any $x\in \mR^d$ and $k\in \mN$, we have
  \begin{eqnarray*}
    \mE\[V(X_{kt_*}(x))\]\leq \alpha^kV(x)+\beta \frac{1}{1-a},
  \end{eqnarray*}
  where  $\alpha,\beta $  and $t_*$ are as in  Hypothesis $\textbf{LC}.$

(\romannumeral2)  There exist    positive constants  $R_*,\beta_*$ and $\alpha_*\in(0,1)$ such that for any $k\in\mN,$
\begin{eqnarray*}
\mE\[V(X_{kt_*}(x))+V(X_{kt_*}(y))\]& \leq & \alpha_* \[V(x)+V(y)\], \ \ \text{for  }~~|(x,y)|\geq R_*,
\\   \mE\[V(X_{kt_*}(x))+V(X_{kt_*}(y))\]& \leq& \beta_*, \ \ \ \ \ \  \ \ \ \ \ \  \text{for  }~~ |(x,y)|\leq R_*.
\end{eqnarray*}
\end{lem}
\begin{proof}
With the help of   Hypothesis  $\textbf{LC},$ we obtain
    \begin{eqnarray}
    \nonumber  \mE\[V(X_{kt_*}(x))\]&=&  \mE\[ \mE\big[V(X_{kt_*}(x))|\cF_{(k-1)t_*}\big]\]
    \\ \nonumber   &\leq & \alpha \mE \[V(X_{(k-1)t_*}(x))\]+\beta,
    \\ \nonumber  &\leq & \alpha \[ a\mE \big[V(X_{(k-2)t_*}(x))\big]+\beta\]+\beta
    \\ \nonumber   &=& \alpha^2 \mE \big[V(X_{(k-2)t_*}(x))\big]+\beta (\alpha+1)
    \\ \nonumber  &\leq &\cdots
    \\ \nonumber   &\leq &  a^k \mE \big[V(X_{(k-k)t_*}(x))\big]+\beta(\alpha^{k-1}+\alpha^{k-2}+\cdots+1)
     \\ \nonumber   &= &  \alpha^k V(x)+\beta (\alpha^{k-1}+\alpha^{k-2}+\cdots+1)
     \\  \label{35}&\leq & \alpha^kV(x)+\beta \frac{1}{1-\alpha},
  \end{eqnarray}
  which gives  (\romannumeral1).

(\ref{35}) also implies that,
\begin{eqnarray*}
&& \mE\[V(X_{kt_*}(x))\]+\mE\[V(X_{kt_*}(y))\]\leq \alpha^k(V(x)+V(y))+\beta \frac{2}{1-\alpha}.
\end{eqnarray*}
Noticing that  $\alpha\in (0,1)$ and
\begin{eqnarray*}
  \lim_{|u|\rightarrow \infty}V(u)=+\infty,
\end{eqnarray*}
we get  the desired    result  (\romannumeral2).
\end{proof}

  \begin{lem}
\label{2}
Let      $R_*$   be as in     Lemma \ref{23}.
There exists $T_0=T_0(R_*)>1$ such that
for any $x,y\in \overline{B}_{R_*}$ and $t>T_0$,
  \begin{eqnarray}\label{31}
\|P_t(x,\cdot)-P_t(y,\cdot)\|_{var}=\sup_{0\leq f\leq 1}\[ P_tf(x)-P_tf(y)\]<1.
  \end{eqnarray}
Moreover, we have
    \begin{eqnarray}
 \label{38}p:&=&\sup_{|x|\leq R_*,|y|\leq R_*}\|P_t(x,\cdot)-P_t(y,\cdot)\|_{var}<1.
 \end{eqnarray}
\end{lem}
\begin{proof}
Since   (\ref{38})
can be  directly obtained by     Hypothesis    $\textbf{H}_2$, (\ref{31})  and  the compactness of the set $\{(x',y')\in \mR^d\times \mR^d: |x'|\leq R_*, |y'|\leq R_*\} $,
we only prove   (\ref{31})  here.

By Hypothesis  $\textbf{H}_1$,  for any $\delta>0$,    there exist $R_0=R_0(R_*),T_0=T_0(R_*,\delta)>1$ such that  for any $x,y\in \overline{B}_{R_*}$ and $t>T_0$
  \begin{eqnarray}\label{47}
  \sup_{\pi\in \Gamma(P_{t-1}^*\delta_{x}, P_{t-1}^*\delta_{y})}\pi\left\{(u,v)\in \mR^d\times \mR^d, |u-v| \leq \delta, u,v\in \overline{B}_{R_0}\right\}>0.
\end{eqnarray}
We emphasis that  $R_0$ is independent of $\delta,$ and the constant  $\delta$
will be given in the next paragraph.

  By  Hypothesis  $\textbf{H}_2$, there exists a constant   $\delta>0$, such that
     \begin{eqnarray}\label{24}
\sup_{0\leq f\leq 1}\[P_1f(u)-P_1f(v)\]\leq \frac{1}{4}
  \end{eqnarray}
  holds for  any  $u,v\in \overline{B}_{R_0}$ with $|u-v|<\delta$.

For any
$\pi \in \Gamma(P_{t-1}^*\delta_x, P_{t-1}^*\delta_y)$ and  $f\in \sB_b(\mR^d)$ with $
\|f\|_\infty \leq 1,$  we have
  \begin{eqnarray}
   && \nonumber  P_tf(x)-P_tf(y)
   \\\nonumber  &&=  \int_{\mR^d\times \mR^d}\pi(\dif u,\dif v)\[P_{1}f(u)-P_{1}f(v)\]
   \\ \nonumber &&= \int_{\Theta}\pi(\dif u,\dif v)  \big[P_{1}f(u)-P_{1}f(v)\big]+\int_{\Theta^c}\pi(\dif u,\dif v)  \big[P_{1}f(u)-P_{1}f(v)\big]
   \\ \label{x-8}&&\leq \frac{1}{2}\pi(\Theta)+\pi(\Theta^c)
  \end{eqnarray}
  where
  $\Theta=\{(u,v)\in \mR^d\times \mR^d, |u-v| \leq \delta, u,v\in B_{R_0}\},$
  $\Theta^c=(\mR^d\times \mR^d)\setminus \Theta.$
Combining    (\ref{47})-(\ref{x-8}),
we obtain  (\ref{31}).

\end{proof}

 Let $T=\[\frac{T_0+1}{t_*}+1\]\cdot t_*:=k_0\cdot t_*$,   where $[b]$ denotes an integer with   $b-1< [b]\leq b,$ $t_*$  is as   in  Hypothesis  $\textbf{LC}$ and $T_0$ is given by Lemma \ref{2},
 then $T>T_0+1$. We  denote   by $P_{x,y}(\cdot)$   the law of   the maximal coupling of $P_T(x,\cdot)$ and $P_T(y,\cdot).$
%It follows from the Skorokhod representation theorem that
Then, there exist a stochastic basis
 $(\ti{\Omega},\ti{\cF},\ti{\mP})$ and  on this   basis,    a  $\mR^d\times \mR^d$
valued Markov chain $\{S(k)\}_{k\geq 0}$ with transition probability family
$\{P_{x,y}(\cdot), {(x,y)\in \mR^d\times \mR^d}\} $.
%  and, on this basis
% There exist  a filtered  probability space  $(\ti{\Omega},\ti{\cF},(\ti{\cF}_t)_{t\geq 0}$
% $,\ti{\mP})$
%and  a  $\mR^d\times \mR^d$
%valued Markov chain $\{S(k)\}_{k\geq 0}$ on this space
%
% \
%
%  There exist  a probability space $(\tilde{\Omega}, \tilde{\cF}, \tilde{\mP})$ and an $\mR^d\times \mR^d$
%valued Markov chain $\{S(k)\}_{k\geq 0}$ on $(\tilde{\Omega}, \tilde{\cF}, \tilde{\mP})$
%with transition probability family
%$P_{x,y}(\cdot)_{(x,y)\in \mR^d\times \mR^d}  $.
%
 Moreover, for every $(x,y)\in \mR^d\times \mR^d$, the marginal chains
$\{S^x(k)\}_{k\geq 0}$ and  $\{S^y(k)\}_{k\geq 0}$ have  the same distribution as $\{X_{kT}(x)\}_{k\geq 0}$ and  $\{X_{kT}(y)\}_{k\geq 0}$,  respectively.

The sequence   $(S(k))_{k\geq 0}$ constructed above is a Markov chain on the  probability space $(\tilde{\Omega},\tilde{\cF},\tilde{\mP})$
which is not necessarily the  same as $(\Omega,\cF,\mP).$ Without loss of generality, we assume that
\begin{eqnarray*}
 (\tilde{\Omega},\tilde{\cF}, \tilde{\mP})=(\Omega,\cF, \mP).
\end{eqnarray*}
Otherwise, we can consider the product space $(\tilde{\Omega}\times \Omega, \cF\times \tilde{\cF}, \tilde{\mP}\times \mP)$.

Define
\begin{eqnarray*}
 \tau_{x,y}&=&\inf\{m\geq 0, S^x(m)=S^y(m)\}.
\end{eqnarray*}
Let $\tau_0=0$ and  for any  $k\geq 1$, we define the stopping time  $\tau_k$ recursively by
\begin{eqnarray*}
  \tau_{k}&=&\inf\{m >\tau_{k-1}, |(S^x(m), S^y(m))|\leq R_* \}.
\end{eqnarray*}
%\begin{eqnarray*}
%\tau_1&=&\inf\{m >0, |(S^x(m), S^y(m))|\leq R_* \},
%  \\   \tau_{k+1}&=&\inf\{m >\tau_k, |(S^x(m), S^y(m))|\leq R_* \},
%  \\ \tau_{x,y}&=&\inf\{m\geq 0, S^x(m)=S^y(m)\}.
%\end{eqnarray*}
% recursively
In the following lemmas,  we will  give some estimates  on these stopping times.
\begin{lem}
\label{36}
For some positive  constant $\theta,$ we have
  \begin{eqnarray*}
    \mE\[e^{\theta \tau_1}\] & \leq  & C(1+V(x)+V(y)),
    \\  \mE\[e^{\theta \tau_k}\] &\leq & C^k\[  1+V(x)+V(y)  \].
  \end{eqnarray*}
\end{lem}
\begin{proof}
Referring to  Lemma \ref{31}  and  \cite[Proposition 3.1]{AS},    we see that  there exists $\theta>0$ such that
\begin{eqnarray}\label{x-9}
  \mE\exp(\theta \tau_1)\leq C(1+V(x)+V(y)).
\end{eqnarray}
Let $u_x=S^x(\tau_{k-1})$ and $ u_y=S^y(\tau_{k-1})$.
By  the  strong Markov property and (\ref{x-9}),  we get
\begin{eqnarray*}
\mE\exp(\theta \tau_k)&=&\mE \[ \mE\exp(\theta \tau_k)\big| \cF_{\tau_{k-1}}
   \]
  \\
   &=&  \mE \[ \mE\exp(\theta \tau_k-\theta \tau_{k-1}+\theta \tau_{k-1})\big| \cF_{\tau_{k-1}}
   \]
   \\
   &=&  \mE\Bigg[\exp(\theta \tau_{k-1})\cdot   \mE\[\exp(\theta \tau_k-\theta \tau_{k-1})\big| \cF_{\tau_{k-1}}  \] \Bigg]
   \\ &\leq & \mE\[\exp(\theta \tau_{k-1})\cdot \(C+CV(u_x)+CV(u_y)\) \]
   \\ &\leq & C \mE\[\exp(\theta \tau_{k-1}) \],
\end{eqnarray*}
which gives the desired result.
\end{proof}

\begin{lem}
For any $k\in \mN,$ we have
  \begin{eqnarray*}
    \mP(\tau_{x,y}>\tau_k+1)\leq p^{k},
  \end{eqnarray*}
  where $p$ appears in  Lemma \ref{2}.
\end{lem}
\begin{proof}
By Lemmas \ref{3}, \ref{2}, we deduce that
\begin{eqnarray*}
 \mP(\tau_{x,y}> \tau_k+1) &\leq & \mP(S^x(k)\neq S^y(k),k=0,\cdots,\tau_k+1)
 \\ &=&  \mE\[  \mP(S^x(k)\neq S^y(k),k=0,\cdots,\tau_k+1| S^x(k), S^y(k), k=0,\cdots, \tau_k\]
 \\ &=&\mE\[I_{\{S^x(k)\neq S^y(k),k=0,\cdots,\tau_k\}}\mP\(S^x(\tau_k+1)\neq S^y(\tau_k+1)|S^x(\tau_k), S^y(\tau_k)\)\]
 \\&\leq & p\mP(S^x(k)\neq S^y(k),k=0,\cdots,\tau_k)
  \\&\leq &  p \mP(S^x(k)\neq S^y(k),k=0,\cdots,\tau_{k-1}+1)
  \\ &\leq &\cdots
  \\&\leq &  p^{k}.
 \end{eqnarray*}
\end{proof}

\begin{lem}
\label{1}
There exists   a constant  $\eps>0$, such that for any $x\neq y, $
\begin{eqnarray*}
\mE\[e^{\eps \tau_{x,y}}\]\leq C(1+V(x)+V(y)).
\end{eqnarray*}
\end{lem}
\begin{proof}
Denote  $\tau_{-1}\equiv-1.$
For any $\eps>0$ and $p'>1,q'>1$ with $\frac{1}{p'}+\frac{1}{q'}=1,$ we have
\begin{eqnarray*}
\mE\[e^{\frac{\eps}{p'} \tau_{x,y}}\]&=& \sum_{k=0}^{\infty}\mE\[e^{\frac{\eps}{p'} \tau_{x,y}}I_{\{\tau_{k-1}+1< \tau_{x,y}\leq \tau_k+1\}}\]
\\ &\leq & \sum_{k=0}^{\infty}\[ \mE  e^{ \eps (\tau_{k}+1)}\]^{1/p'}\mP\(\tau_{x,y}>\tau_{k-1}+1 \)^{1/q'}
\\ &\leq & \sum_{k=0}^{\infty} C \[ \mE  e^{ \eps (\tau_{k}+1)}\]^{1/p'}  p^{k/q'}.
\end{eqnarray*}
Setting    $\eps \leq \theta$ and $p'$ big enough,  by Lemma \ref{36}    we derive  that
\begin{eqnarray*}
\mE\[e^{\frac{\eps}{p'} \tau_{x,y}}\]&\leq & \sum_{k=1}^{\infty}  C\cdot C_1^{k/p'}\[  1+CV(x)+CV(y)  \]^{1/p'}\cdot  p^{k/q'}
\\ &\leq & C(1+V(x)+V(y)).
\end{eqnarray*}
\end{proof}

\subsection{Proof of Theorem \ref{9} }
\begin{proof}
By Lemma \ref{1}, for some $\theta>0$ and any   $ f\in \sB_b(\mR^d)$ with $\|f\|_{\infty}\leq 1$, we obtain
\begin{eqnarray*}
  \mE\[f(S^x(k))-f(S^y(k))\]\leq C e^{-\theta k}(1+V(x)+V(y)).
\end{eqnarray*}
Since the marginal chains   $\{S^x(k)\}_{k\geq 0},$ $\{S^y(k)\}_{k\geq 0}$ have  the same distribution as $\{ X_{kT}(x) \}_{k\geq 0}$,$\{ X_{kT}(y) \}_{k\geq 0}$  respectively,     we have
\begin{eqnarray}\label{46}
  \mE\[f( X_{kT}(x) )-f( X_{kT}(y) )\]\leq C e^{-\theta k}(1+V(x)+V(y)).
\end{eqnarray}
For any $t>0$ and $f$ with  $\|f\|_\infty\leq 1$,  we  set $k_t=[\frac{t}{T}],\ti{f}=P_{t-k_tT}f$.
By (\ref{46}),  we get
\begin{eqnarray*}
P_tf(x)-P_tf(y)&=& \mE\[P_{t-k_tT}f( X_{k_tT}(x))-P_{t-k_tT}f(X_{k_tT}(y))\]
\\ &=& \mE\[\ti{f}(X_{k_tT}(x))-\ti{f}(X_{k_tT}(y))\]
\\ &\leq &  C e^{-\theta k_t}(1+V(x)+V(y))
\\  &\leq &  C e^{-\theta \frac{t}{T}}(1+V(x)+V(y)).
\end{eqnarray*}
Hence, for some $\theta>0,$ we obtain
\begin{eqnarray}\label{15}
P_tf(x)-P_tf(y)\leq  C\[1+V(x)+V(y)\]e^{-\theta t}.
\end{eqnarray}
According to  \cite[Section 2.2]{AS},  (\ref{15}) implies  (\ref{16}) which finishes the proof of   Theorem \ref{9}.
For the convenience of reading, we still give  its  details here.

Denote  by $\langle f, P_t(x,\cdot)\rangle=P_tf(x)$.
For any $s>t$,  it is not difficult to see that
\begin{eqnarray}
\nonumber
 |\langle f, P_t(x,\cdot)-P_s(y,\cdot)\rangle|&=&\left| \int_{\mR^n}P_{s-t}(y,dz)\int_{\mR^n}\(P_t(x,dv)-P_t(z,dv)\)f(v)\right|
 \\
 \nonumber
 &\leq &Ce^{-\theta t} \int_{\mR^n}P_{s-t}(y,dz)(1+V(x)+V(z))
 \\
 \label{5}
  &=& Ce^{-\theta t}\(1+V(x)+\mE_y \big[V(X_{s-t})\big]\).
\end{eqnarray}
By the Prokhorov theorem, $\cP(\mR^d)$ is a complete metric space under the weak topology (here $\cP(\mR^d)$ denotes the set of all probability measures on  $\mR^d$).
Let $y=0$ and  $s\rightarrow +\infty$ in  (\ref{5}),
  % for any  $f\in C_b(\mR^d)$ with $\|f\|_\infty \leq 1$,
one arrives at that
\begin{eqnarray*}
  |\langle f, P_{t}(x,\cdot)-\mu\rangle|\leq Ce^{-\theta t}(1+V(x)),
\end{eqnarray*}
which implies  (\ref{16}).
\end{proof}
% \section{A  proof of Theorem \ref{29}}

\section{Application}
 Consider the following stochastic differential equation driven by L\'evy Processes
\begin{eqnarray}\label{10}
\dif X_t = b(X_t)\dif t + A_1\dif W_t + A_2\dif L_t
, \ X_0 = x\in \mR^d,
\end{eqnarray}
where $b : \mR^d \rightarrow \mR^d$
is a smooth vector field, $A_1$ and $A_2$ are two constant $d \times d$-matrices, $(W_t)_{t\geq 0}$
is a
d-dimensional standard Brownian motion and $(L_t)_{t\geq 0}$
is a purely jump d-dimensional L\'evy process
with L\'evy measure $\nu(dz)$. Let $\Gamma_0 := \{z \in \mR^d
: 0 < |z| < 1\}$.

Throughout this section, we assume
that $\left.\frac{\nu(dz)}{dz}\right|_{\Gamma_0} = \kappa(z)$ satisfies the following conditions: for some $\alpha\in (0, 2)$ and $m\in\mN$,\\
$(\textbf{H}_m^\alpha):$
$ \kappa\in C^m(\Gamma_0,(0,\infty))$ is symmetric (i.e. $\kappa(-z) = \kappa(z)$) and satisfies the following Orey’s order condition (cf. \cite[Proposition 28.3]{Sato})
\begin{eqnarray*}
  \lim_{\eps\rightarrow 0} \eps^{\alpha-2}\int_{|z|\leq \eps}|z|^2\kappa(z)\dif z:=c_1>0,
\end{eqnarray*}
and bounded condition: for $j = 1,\cdots, m $ and some $C_j > 0,$
$$
|\nabla^j \log \kappa(z)|\leq C_j |z|^{-j}, z\in \Gamma_0.
$$

Let $A^*$ be  the transpose of $A$, and
$$
\nabla^2_{A_1A_1^*}f:=\sum_{i,j=1}^d (A_1A_1^*)_{ij}\partial_{ij}^2f,
$$
Let $B_0=\mI_{d\times d}$ be the identity matrix and define for $n \in \mN$,
\begin{eqnarray*}
  B_n(x):= b(x)\cdot \nabla B_{n-1}(x) - \nabla b(x) \cdot B_{n-1}(x) +\frac{1}{2}\nabla^2_{A_1A_1^*}B_{n-1}(x).
\end{eqnarray*}
Here and below $(\nabla b)_{ij}:=\partial_jb^i(x).$

Let $P_t(x,\cdot)$ be the transition probability associated with Eq.(\ref{10}).
We  claim that the following theorem holds.
\begin{thm}\label{29}
Assume $(\textbf{H}_1^\alpha)$  holds,  $\int_{\mR^d}|z|^2\nu(\dif z)<\infty$, and
\begin{itemize}
  \item[(1)] for some  $k>0,$
\begin{eqnarray}\label{51}
\langle x-y, b(x)-b(y)\rangle\leq -k|x-y|^2,
\end{eqnarray}
where $\langle \cdot,\cdot \rangle$ denotes the usual inner product on $\mR^d,$
  \item[(2)]   for any $x\in\mR^d,$ there exists some  $n=n(x)$ such that
  \begin{eqnarray}\label{52}
    \text{Rank}\[A_1,B_1(x)A_1,\cdots,B_n(x)A_1,A_2,B_2(x)A_2,\cdots,B_n(x)A_2\]=d,
  \end{eqnarray}
\end{itemize}
 then there exist  a  unique invariant probability measure $\mu$ for $P_t$ and a  positive   constant $\theta$ such that for any $x\in \mR^d$,
\begin{eqnarray*}
  \|P_t(x,\cdot)-\mu\|_{var}\leq Ce^{-\theta t}(1+|x|^2), \ \ \ \forall t\geq 0.
\end{eqnarray*}
\end{thm}

Before  we give a proof of Theorem  \ref{29}, we give some remarks, notations and a lemma.

\begin{Rem}
According to \cite[Theorem 1.1]{SZ}, we know that if $(\textbf{H}_1^\alpha)$ and (\ref{52}) hold,  Hypothesis  $\textbf{H}_1$  also holds. Moreover, (\ref{51}) is used to ensure Hypotheses  $\textbf{LC}$, $\textbf{H}_2$ hold.
\end{Rem}
\begin{Rem}
When  $\|A_1^{-1}\|\leq \lambda $
for some $\lambda>0,$
here $A_1^{-1}$ denotes the inverse of $A_1$, Lan and Wu \cite[Theorem 1.3]{LW} proved that $X_t$ is irreducible aperiodic. However, generally speaking,  (\ref{52}) doesn't imply  $\textbf{H}_a.$
For this, one can see Example \ref{E-1}.
\end{Rem}

Let  $N$ be  a Poisson  random measure with density $\dif t\nu(\dif z)$, where $\nu(\dif z)=\kappa(z)\dif z$ and
let
\begin{align*}
 &  L_t=\int_0^t\int_{|z|<1}z\tilde{N}(\dif s\dif z)+\int_0^t\int_{|z|\geq 1}zN(\dif s\dif z),
  \\ &\tilde{N}(\dif t\dif z)=N(\dif t\dif z)-\dif s\nu(\dif z).
\end{align*}
Then, (\ref{10}) can be rewritten as
\begin{eqnarray*}
\left\{
\begin{split}
\dif X_t &=  b(X_t)\dif t + A_1\dif W_t + \int_{\mR^d}A_2z\tilde{N}(\dif s\dif z)+\int_{\mR^d}A_2zN(\dif s\dif z),
\\
X_0 &= x.
\end{split}
\right.
\end{eqnarray*}

For any $\ell\in\mN,$ let $h_\ell(x):\mR^d\rightarrow  \mR$ be a smooth function  with compact support and $h_\ell(x)=1$ when  $x\in \overline{B}_\ell.$ Let $X_t^\ell(x)$ be the solution to the following SDEs
\begin{eqnarray}\label{x-5}
\dif X_t^\ell =h_\ell(X_t^\ell) b(X_t^\ell)\dif t + A_1\dif W_t + A_2\dif L_t
, \ X^\ell_0 = x.
\end{eqnarray}
Following   the proof of \cite[Theorem 1.1]{SZ}, one obtains that for any $x\in B_\ell,$
\begin{eqnarray}\label{z-2}
  \lim_{y\rightarrow x}\sup_{\|f\|_{\infty}\leq 1}\mE\[f(X_t^\ell(y))-f(X_t^\ell(x))\]=0.
\end{eqnarray}
Define a sequence of stopping time
\begin{eqnarray*}
\tau_\ell(x)&=&\inf\{s>0,~X^\ell_s(x)\not\in B_\ell\},\ \ell\in \mN, x\in \mR^d,
\end{eqnarray*}
then the following lemma holds.

\begin{lem}
For any $x \in \mathbb{R}^{d},$ $l\geq 2,~t>0$, we have
\begin{eqnarray}\label{25}
\limsup_{y\rightarrow x} I_{\{t>\tau_\ell(x)\}}\leq I_{\{t\geq
\tau_{\ell-1}(x)\}},~ a.s.
\end{eqnarray}
\end{lem}
\begin{proof}
%Let $\Gamma$ be a measurable set with $\mathbb{P}(\Gamma^c)=0$ and such that $X^{l}_s(x,\omega)$ is c\`{a}dl\`{a}g with respect to  $s$ for $\omega\in\Gamma$. For $\omega \in \Gamma$, t
The conclusion is apparent if
\begin{eqnarray*}
\limsup_{y\rightarrow x}I_{\{t>\tau_\ell(y)\}}(\omega)=0\text{ or }I_{\{t\geq \tau_{\ell-1}(x)\}}(\omega)=1.
\end{eqnarray*}
Assume that $\limsup\limits_{y\rightarrow x}I_{\{t>\tau_\ell(y)\}}(\omega)=1$ and $I_{\{t\geq \tau_{\ell-1}(x)\}}(\omega)=0$,
then
\begin{align*}
\sup_{s \in [0,t]} |X^{\ell-1}_s(x,\omega)| \leq  \ell-1.
\end{align*}
Therefore,
\begin{eqnarray}\label{26}
\sup_{s \in [0,t]} |X^{\ell}_s(x,\omega)| \leq  \ell-1.
\end{eqnarray}
Since $\limsup\limits_{y\rightarrow x}I_{\{t>\tau_\ell(y)\}}(\omega)=1$,   there exists  a sequence  $\{x_n\}\subset\mathbb{R}^{d}$ with
$x_n\rightarrow x$ as $n\rightarrow\infty$, such that for $n$ large enough
\begin{align}\label{27}
\sup_{s \in [0,t]}|X^{l}_s(x_{n},\omega)|\geq l.
\end{align}
On the other hand, by (\ref{x-5}),  we have
\begin{eqnarray*}
  \dif |X_t^\ell(x_n)-X_t^\ell(x)|\leq C |X_t^\ell(x_n)-X_t^\ell(x)|\dif t,
\end{eqnarray*}
and furthermore
\begin{eqnarray*}
  \sup_{s\in [0, t]} |X_s^\ell(x_n)-X_s^\ell(x)|\leq e^{Ct}|x_n-x|,
\end{eqnarray*}
which contradicts  (\ref{26})  and (\ref{27}).
\end{proof}

Now we are in a position  to give the proof of Theorem  \ref{29}.

\begin{proof}
Based on  Theorem  \ref{9}, we only need to verify  Hypotheses  $\textbf{LC}$, $\textbf{H}_1$ and $\textbf{H}_2$  respectively.

(\Rmnum{1}) Verification of    Hypothesis $\textbf{LC}$:
Using   It\^o's formula, we  have
\begin{eqnarray*}
  |X_t(x)|^2&=&|x|^2+2\int_0^t \langle  X_{s_-}^x, b(X_{s_-}) \rangle\dif s+2\int_0^t\langle X_{s_-}^x, A_1\dif W_s\rangle+\int_0^tTr(A_1\cdot A_1^*)\dif s
  \\ &&+\int_0^t\int_{|z|< 1}\[|X_{s_-}^x+A_2z|^2-|X_{s_-}^x|^2 \]\tilde{N}(\dif s\dif z)
  +\int_0^t\int_{|z|\geq  1}\[|X_{s_-}^x+A_2z|^2-|X_{s_-}^x|^2 \]N(\dif s\dif z)
  \\ &&+\int_0^t\int_{|z|< 1}\[ |X_{s_-}^x+A_2z|^2-|X_{s_-}^x|^2 -\sum_{i=1}^d(A_2z)_iX_i^x(s_-) \]\dif s\nu(\dif z).
\end{eqnarray*}
For any $\eps>0$,  one easily sees that
\begin{align}\label{28}
  \begin{split}
\langle X_{s_-}^x, A_2z \rangle& \leq \eps | X_{s_-}^x |^2+C(\eps) |z|^2, \forall \eps>0,
\\
\langle b(x),x\rangle&=\langle b(x)-b(0),x-0\rangle+\langle b(0),x\rangle\leq -k|x|^2+|b(0)|\cdot |x|.
  \end{split}
\end{align}
Then, we deduce that
\begin{eqnarray*}
|X_t(x)|^2
 &\leq &  |x|^2+2\int_0^t \[-k|X_{s_-}^x|^2+|b(0)|\cdot |X_{s_-}^x|\]\dif s+2\int_0^t\langle X_{s_-}^x, A_1\dif W_s\rangle+\int_0^tC(\eps) \dif s
\\ && +\int_0^t\int_{|z|< 1}\[|X_{s_-}^x+A_2z|^2-|X_{s_-}^x|^2 \]\tilde{N}(\dif s\dif z)
  +\int_0^t\int_{|z|\geq  1}\[|X_{s_-}^x+A_2z|^2-|X_{s_-}^x|^2 \]N(\dif s\dif z)
   \\ &&+\int_0^t\int_{|z|< 1}\[ \eps |X_{s_-}|^2+C(\eps)|z|^2\]\dif s\nu(\dif z).
\end{eqnarray*}
Setting  $\eps$ small enough and  combining  the above inequality   with (\ref{28}),  we deduce  that  for some $C=C(k),$
\begin{eqnarray*}
 \mE\[|X_t(x)|^2\] &&\leq  |x|^2+\int_0^t \[-\frac{3k}{2}\mE |X_{s_-}^x|^2+C\]\dif s
 +\int_0^t\int_{|z|\geq  1}\mE\[|X_{s_-}^x+A_2z|^2-|X_{s_-}^x|^2 \]\dif s\nu(\dif z)
\\ &&\leq  |x|^2+\int_0^t \[-\frac{3k}{2}\mE |X_{s_-}^x|^2+C\]\dif s
 +\int_0^t\int_{|z|\geq  1}\mE\[\langle X_{s_-}^x, A_2z\rangle +C|z|^2\]\dif s\nu(\dif z)
\\ && \leq   |x|^2+\int_0^t \[-\frac{3k}{2}|X_{s_-}^x|^2+C\]\dif s
 +\int_0^t\int_{|z|\geq  1}\[\frac{k}{2} \mE | X_{s_-}^x|^2+C|z|^2\] \dif s\nu(\dif z)
\\ && \leq |x|^2 -k\int_0^t \mE |X_{s_-}^x|^2 \dif s +\int_0^t C \dif s.
\end{eqnarray*}
By the   Gronwall's  inequality,  we obtain      that
\begin{eqnarray}\label{50}
\mE\[|X_t(x)|^2\]\leq |x|^2e^{-kt}+C,
\end{eqnarray}
which gives the desired result.

(\Rmnum{2}) Verification of    Hypothesis $\textbf{H}_1$:
By  calculating, we get
\begin{eqnarray*}
  \dif |X_t(x)-X_t(y)|^2&=&2 \langle b(X_t(x))-b(X_t(y)), X_t(x)-X_t(y)\rangle \dif t
  \\ &\leq & -2k|X_t(x)-X_t(y)|^2\dif t,
\end{eqnarray*}
which implies that
\begin{eqnarray}\label{z-1}
 |X_t(x)-X_t(y)|^2\leq |x-y|^2e^{-kt}.
\end{eqnarray}
In view  of  (\ref{50}) and (\ref{z-1}), Hypothesis $\textbf{H}_1$ holds.

(\Rmnum{3}) Verification of    Hypothesis $\textbf{H}_2$:
 The main  ideas in this verification are  borrowed    from the proof of \cite[Theorem 1.1]{SZ} and    \cite[Theorem 4.2]{DP1}.

 % \textbf{  verification of    Hypothesis $\textbf{H}_2$}:
For any $x,y\in B_\ell$, one  sees  that
\begin{eqnarray*}
&& |\mE[f(X_t(x))-f(X_t(y))]|
\\&&\leq|\mE[f(X_t(x))1_{t<\tau_\ell(x)}-f(X_t(y))1_{t<\tau_\ell(y)}]|+\|f\|_\infty\mP(t\geq \tau_\ell(x))+\|f\|_\infty\mP(t\geq \tau_\ell(y))
\\
&&\leq|\mE[f(X_t^\ell(x))1_{t<\tau_\ell(x)}-f(X_t^\ell(y))1_{t<\tau_\ell(y)}]|
+\|f\|_\infty\mP(t\geq \tau_\ell(x))+\|f\|_\infty\mP(t\geq \tau_\ell(y))
\\
&&\leq |\mE[f(X_t^\ell(x))-f(X_t^\ell(y))]|
+2\|f\|_\infty\mP(t\geq \tau_\ell(x))+2\|f\|_\infty\mP(t\geq \tau_\ell(y)).
\end{eqnarray*}
Thus, by  (\ref{z-2}),  we derive that
\begin{align*}
  \lim_{y\to x}\sup_{\|f\|_\infty\leq 1}|\mE[f(X_t(x))-f(X_t(y))]|&\leq 2\mP(t\geq \tau_\ell(x))+2\limsup_{y\rightarrow x}\mP(t\geq \tau_\ell(y))
  \\ & \leq 2\mP(t\geq \tau_\ell(x))+2\mP(t\geq \tau_{\ell-1}(x)).
\end{align*}
Letting  $\ell\rightarrow \infty$ in the above inequality and by (\ref{25}), we finish  the  verification of    Hypothesis $\textbf{H}_2$.

\end{proof}

\vspace{5mm}

{\bf Acknowledgements:}
 Special thanks are due to the referees, professor Zhao, Dong  and professor  Xicheng, Zhang for carefully checking details of the paper and
 helping to improve the new version.

% \section*{Acknowledges}

\end{document}